\documentclass[11pt]{article}
\usepackage[margin=1.0in]{geometry}
\usepackage{amsthm}
\usepackage{amssymb,amsfonts,amsmath,latexsym}
\usepackage{graphicx}
\usepackage{xcolor}
\usepackage[colorlinks = true,
            linkcolor = blue,
            urlcolor  = blue,
            citecolor = blue,
            anchorcolor = blue]{hyperref}

\newcounter{count}
\theoremstyle{plain}

\newtheorem{lemma}[count]{Lemma}
\newtheorem{theorem}[count]{Theorem}
\newtheorem{corollary}[count]{Corollary}

\usepackage[symbol]{footmisc}

\newcommand{\ZZ}{\mathbb{Z}}

\usepackage[
backend=biber,
citestyle=numeric,
bibstyle=numeric,
sorting=nty,
autocite=superscript
]{biblatex}

\usepackage{mathrsfs}

\title{An exact test for significance of clusters in binary data}
\author{James Mathews\thanks{Memorial Sloan-Kettering Cancer Center}, Cameron Crowe, Rami Vanguri \textsuperscript{*},\\ Margaret Callahan\textsuperscript{*}, Travis Hollmann\textsuperscript{*}, Saad Nadeem\textsuperscript{*}}
\begin{document}

\maketitle

\begin{abstract} Unsupervised clustering of feature matrix data is an indispensible technique for exploratory data analysis and quality control of experimental data. However, clusters are difficult to assess for statistical significance in an objective way. We prove a formula for the distribution of the size of the set of samples, out of a population of fixed size, which display a given signature, conditional on the marginals (frequencies) of each individual feature comprising the signature. The resulting ``exact test for coincidence'' is widely applicable to objective assessment of clusters in any binary data. We also present a software package implementing the test, a suite of computational verifications of the main theorems, and a supplemental tool for cluster discovery using Formal Concept Analysis.
\end{abstract}

\tableofcontents

\newpage

\section{Introduction}
A typical visualization of a binary data matrix is a hierarchically-clustered heatmap, with dendrograms in which the higher-level clusters are recursively comprised of smaller clusters, the hierarchy being computed with an agglomeration strategy involving a distance function defined pairwise between samples (or features) to be clustered. In favorable cases a cluster may appear at some level of the hierachy which is especially characteristic of an important underlying state or measure, i.e. an outcome. For example, likelihood of favorable response to some medical treatment.

But it is often difficult to decide whether a cluster found this way, or any other way, could just as easily have occurred by random chance. This is obviously a primary concern in the unsupervised context, where outcomes which might guide cluster assessment are not present. It is also a concern in the supervised context, due to the possibility of overfitting or multiple-hypothesis false discovery.

As an example, in recent work of the authors\autocite{psgpaper}, a subtype of several types of cancers (including lung and uterus cancers) was identified which exhibited a molecular signature defined by about 10 genes, the PSGs. Network analysis methods implicated the gene subset, but initially confidence concerning its actual significance was low. Pearson correlation analysis was inconclusive due to the presence of outliers. The rarity of the subtype displaying the full signature added to this uncertainty. Ultimately Kaplan-Meier analysis did show that the PSG+ phenotype confers a poor prognosis, confirming the biological significance of this subtype, but we still lacked an objective basis for any claim of statistical significance of the signature/subtype itself. The exact test we introduce in this article turns out to provide such a basis, as described in Figure \ref{psgfigure}.

\section{Theory}
\subsection{Setup}
Let $\overline{M}$ be a binary matrix of shape $(N, K)$. We call the $K$ columns \emph{features} and the $N$ rows \emph{samples}. Given a $k$-element subset $F$ of the feature set, let $M$ denote the restriction of $\overline{M}$ to the columns $F$, and let $v:=(v_1, \ldots, v_k)$ denote the corresponding column sums. Let $S$ denote the set of samples which have all of the features $F$. That is,
\begin{align*}
S = \{ \thinspace s \quad \vert \quad M(s,f)=1 \quad \forall f\in F \}
\end{align*}
A pair $(F, S)$ obtained as above may be called a \emph{maximal bicluster} or a \emph{formal concept}. We shall use the term \emph{signature}, emphasizing the feature set $F$, and call $S$ the set of samples \emph{displaying signature} $F$. In appendix \ref{fca_methods} we explain how to identify, in practice, many examples of $(F, S)$ for which $S$ is non-empty and relatively large. Of course, any other signature discovery method may be used instead.

We propose to assess the significance of a given signature finding in terms of the size $|S|$, under the intuition that simultaneous display of multiple features by a large set of samples indicates a non-trivial relation between the features. We call this size the \emph{incidence} or \emph{intersection} statistic, and denote it $I$.

\subsection{Binary matrix configurations}
We are concerned with binary matrices $M$. If $M$ has $k$ columns, it will be convenient to do some calculations in the ring of formal power series $T:=\ZZ[[t_1, \ldots, t_k]]$. This is because of the correspondence between:
\begin{enumerate}
\itemsep0em
    \item{Multiplicity-free monomials in $T$, i.e. elements of the form $t_J:=\prod\limits_{j\in J}t_j$ for some $J\subset \{1,\ldots,k\}$}
    \item{Subsets $J\subset \{1,\ldots, k\}\thinspace$ (i.e. $J\in \mathscr{P}_{k}$)}
    \item{Possible rows $r=(r^{1},\ldots,r^{k})$ of $M$}
\end{enumerate}
The correspondence is
\begin{align*}
t_J \quad \longleftrightarrow \quad J \quad \longleftrightarrow \quad r=(r^{1},\ldots r^{k}),\thinspace r^j = \begin{cases} 1 \mbox{ if } j\in J \\ 0 \mbox{ if } j\notin J\end{cases}
\end{align*}
Denote by $\mathscr{F}(k)$ the set defined by any of these 3 equivalent descriptions. (Here $\mathscr{F}$ stands for "features".)

Symmetrically, if $M$ has $n$ rows, we consider the ring $W:=\ZZ[[s_1,\ldots,s_n]]$, and the 3 sets in correspondence:
\begin{enumerate}
\itemsep0em
    \item{Multiplicity-free monomials in $W$, i.e. elements of the form $s_U:=\prod\limits_{u\in U}s_u$ for some $U\subset \{1,\ldots,n\}$}
    \item{Subsets $U\subset \{1,\ldots, n\}\thinspace$ (i.e. $U\in \mathscr{P}_{n}$)}
    \item{Possible columns $c=(c^{1},\ldots,c^{n})$ of $M$}
\end{enumerate}

Denote this set by $\mathscr{S}(n)$. (Here $\mathscr{S}$ stands for "samples").

In these terms, the set of all $M$ is naturally identified with $(\mathscr{F}(k))^{n}$ and with $(\mathscr{S}(n))^{k}$ by regarding $M$ as an $n$-tuple of rows or, respectively, as a $k$-tuple of columns.

We will also call the matrices $M$ \emph{configurations}, writing 
\[(\mathscr{F}(k))^{n}\cong (\mathscr{S}(n))^{k}=: \mathscr{C}\]
\[(J_1,\ldots,J_n)\quad \longleftrightarrow\quad (U_1,\ldots,U_k) \quad \longleftrightarrow  M\]
In counting configurations satisfying certain conditions, we will appeal to the notation introduced above for corresponding elements in lieu of explicit notation for the bijection functions.

\subsection{Incidence statistic, its PMF, and CDF\label{formuladist}}
Define integers $a(n,v)$, for integers $n\geq 0$ and $v=(v_1,\ldots, v_k)$ with $v_j\geq 0\thinspace \forall j$, by the generating function:
\begin{align*}
(f(t)-t_1\cdots t_k)^{n}& =: \sum_{v} a(n,v) t_1^{v_1}\cdots t_k^{v_k} \\
f(t):&= (1+t_1)\cdots(1+t_k)
\end{align*}
The following counting theorem is the underlying fact needed to prove a formula for the probability mass function (PMF) of the incidence statistic.
\begin{theorem}\label{intersectioncounting}\hfill
\begin{enumerate}
\itemsep0em
    \item{\label{interpretation_intersection}$a(n,v)$ is the number of configurations in which the mutual intersection of the $U_j$ is empty, that is $\cap_{j=1}^{j=k}U_j=\emptyset$, and such that $|U_j|=v_j$ for each $j$.}
    \item{\label{formula_intersection}$a(n,v)=\sum\limits_{m=0}^{m=n}(-1)^{n+m}\binom{n}{m}\prod\limits_{j=1}^{j=m}\binom{m}{n-v_j} $}
\end{enumerate}
\end{theorem}
\begin{proof} (\ref{interpretation_intersection}) By expansion, $f(t)$ consists of the sum of all the monomials in $\mathscr{F}(k)$. So $f(t)-t_1\cdots t_k$ is the sum of all the monomials except $t_1\cdots t_k$. Before collecting terms with the same monomial part, the terms of $(f(t)-t_1\cdots t_k)^n$ are labelled by ordered $n$-tuples of elements of $\mathscr{F}(k)\backslash \{t_1\cdots t_k\}$. That is, by certain elements of $\mathscr{C}$. Thus the notation we have introduced for elements of $\mathscr{C}$ may be brought to bear. In particular, the monomial part of a given term is
\[ t_1^{|U_1|}\cdots t_k^{|U_k|} \]
It follows that the coefficient of $t_1^{v_1}\cdots t_k^{v_k}$ is the number of configurations, in which no $J_i$ equals the whole set $\{1,\ldots,k\}$ (due to the missing element $t_1\ldots t_k$), such that $|U_j|=v_j$ for all $j$. The condition that no $J_i$ be equal to the whole set is equivalent to the mutual intersection of $U_j$ being empty.

\noindent(\ref{formula_intersection}) We apply the binomial theorem $1+k$ times:
\begin{align*}
(f(t)-t_1\cdots t_k)^{n} &= \sum\limits_{m=0}^{m=n} (-1)^{n-m} \binom{n}{m} (f(t))^{m} (t_1^{n-m}\cdots t_k^{n-m}) \\
 &=\sum\limits_{m=0}^{m=n} (-1)^{n+m} \binom{n}{m} (1+t_1)^{m}\cdots(1+t_k)^{m} (t_1^{n-m}\cdots t_k^{n-m})\\
 &=\sum\limits_{m=0}^{m=n} (-1)^{n+m} \binom{n}{m} \left(\sum\limits_{u=0}^{u=m}\binom{m}{u}t_1^{u}\right)\cdots\left(\sum\limits_{u=0}^{u=m}\binom{m}{u}t_k^{u}\right) (t_1^{n-m}\cdots t_k^{n-m})\\
 &=\sum\limits_{m=0}^{m=n} (-1)^{n+m} \binom{n}{m} \left(\sum\limits_{v}\prod\limits_{j=1}^{j=k}\binom{m}{v_j}t_1^{v_1}\cdots t_k^{v_k}\right) (t_1^{n-m}\cdots t_k^{n-m})\\
 &=\sum\limits_{v}\sum\limits_{m=0}^{m=n} (-1)^{n+m} \binom{n}{m} \prod\limits_{j=1}^{j=k}\binom{m}{v_j}t_1^{n-m+v_1}\cdots t_k^{n-m+v_k}\\
 &=\sum\limits_{v}\sum\limits_{m=0}^{m=n} (-1)^{n+m} \binom{n}{m} \prod\limits_{j=1}^{j=k}\binom{m}{v_j-(n-m)}t_1^{v_1}\cdots t_k^{v_k}\\
 &=\sum\limits_{v}\sum\limits_{m=0}^{m=n} (-1)^{n+m} \binom{n}{m} \prod\limits_{j=1}^{j=k}\binom{m}{n-v_j}t_1^{v_1}\cdots t_k^{v_k}
\end{align*}
\end{proof}
The proof above is clarified somewhat by the observation that $f(t)$ can be expressed as a specialization of the power series in $W\otimes_{\mathbb{Z}}T$,
\[g(s,t):=\prod\limits_{u,j}(1+s_u t_j)\quad,\]
namely $f(t)=g(\mathbf{1},t)$ where $\mathbf{1} = (1, \ldots, 1)$.
\begin{theorem} Fix integers $i\geq 0$, $v=(v_1,\ldots,v_k)$, $v_j\geq 0$, and $n>0$. Consider the $n\times k$ configurations $M$ in which:
\begin{enumerate}
\itemsep0em
\item{$|U_j|=v_j$ for each $j$.}
\item{The cardinality of the intersection of the $U_j$ is exactly $i$, that is $|\cap_{j=1}^{j=k}U_j|=i$.}
\end{enumerate}
The number of such configurations is given by the formula:
\[ \binom{n}{i}\sum\limits_{m=0}^{m=n-i}(-1)^{n-i+m}\binom{n-i}{m}\prod\limits_{j=1}^{j=m}\binom{m}{n-v_j}  \]
\end{theorem}
\begin{proof} The indicated set of configurations is partitioned equally into $\binom{n}{i}$ sets, according to which $i$-element sample subset is the mutual intersection, denoted $X$. By construction the reduced configuration matrix, not involving the elements of $X$, must consist of $k$ features with sample sets of sizes $(v_1-i, \ldots, v_k-i)$ and with no intersection. Thus the size of each part of the partition is $a(n-i, (v_1-i,\ldots,v_k-i))$. The number of configurations is therefore
\[\binom{n}{i}a(n-i, (v_1-i,\ldots, v_k-i))\]
The result follows from the formula for $a$ given in Theorem \ref{intersectioncounting}.\ref{formula_intersection}.
\end{proof}

The null assumption we make for our test is the one that is made implicitly in a standard permutation test, namely the uniform distribution on the subset of $\mathscr{C}$ defined by $|U_{j}|=v_j$, given $v=(v_1,\ldots,v_k)$. Note that this entails that we do \emph{not} assume $M$ is comprised of $n$ independent and identically distributed (iid) samples. Also, despite the fact that $M$ appears to be $n$ samples from a set of binary discrete variables, it is definitely not $n$ samples of Bernoulli variables; for example, the variance of the number of positives is 0 for each feature, rather than $np(1-p)$ for some positivity rate $p$.

Under this assumption the incidence statistic $I$ is an integer-valued random variable. The following corollary provides a formula for its PMF.

\begin{corollary}\label{stattest}Consider $n$ samples observed with $k$ binary features of respective frequencies $v_1,\ldots v_k$. The probability of observing exactly $i$ samples positive for all $k$ features is:
\[ p(I=i) = \frac{\binom{n}{i}\sum\limits_{m=0}^{m=n-i}(-1)^{n-i+m}\binom{n-i}{m}\prod\limits_{j=1}^{j=m}\binom{m}{n-v_j}}{\prod\limits_{j=1}^{j=k}\binom{n}{v_j}} \] 
\end{corollary}

By summing over several values of $i$ in Corollary \ref{stattest}, one can compute a value of the cumulative distribution function (CDF) of $I$. This is (one minus) the $p$-value for the proposed ``exact test for coincidence".

The next theorem provides an alternative, more closed-form calculation of the CDF, with significantly decreased computational complexity compared with direct summation of PMF values, namely $O(n)$ rather than $O(n^2)$.

The proof of this theorem depends on two basic lemmas.
\begin{lemma}
\[ \binom{a}{b} \binom{b}{c} = \binom{a-c}{a-b} \binom{a}{c} \]
\end{lemma}
\begin{proof}
\begin{align*}
\frac{a!}{(a-b)!b!}\cdot \frac{b!}{(b-c)!c!} = \frac{1}{(a-b)!(b-c)!}\cdot \frac{a!}{c!} = \frac{(a-c)!}{(a-b)!(b-c)!}\cdot \frac{a!}{(a-c)!c!}
\end{align*}
\end{proof}

\begin{lemma}
\[ \sum\limits_{h=0}^{h=l}(-1)^{h}\binom{g}{h} = (-1)^{l} \binom{g-1}{l} \]
\end{lemma}
\begin{proof}
By induction. Base case $g=1$:
\begin{align*}
(-1)^{0}\binom{1}{0} &= 1 = (-1)^{0}\binom{0}{0} \\
\binom{1}{0} - \binom{1}{1} &= 0 = (-1)^{1}\binom{0}{1}
\end{align*}
Now assume the formula holds (for all $l$) for a fixed $g\geq 0$.
\begin{align*}
\sum\limits_{h=0}^{h=l}(-1)^{h}\binom{g+1}{h} &= \sum\limits_{h=0}^{h=l}(-1)^{h}\left(\binom{g}{h}+\binom{g}{h-1}\right) \\
 &= \sum\limits_{h=0}^{h=l}(-1)^{h}\binom{g}{h} + \sum\limits_{h=0}^{h=l}(-1)^{h}\binom{g}{h-1} \\
 &= \sum\limits_{h=0}^{h=l}(-1)^{h}\binom{g}{h} + \sum\limits_{h=1}^{h=l}(-1)^{h}\binom{g}{h-1}\\
 &= \sum\limits_{h=0}^{h=l}(-1)^{h}\binom{g}{h} - \sum\limits_{h=0}^{h=l-1}(-1)^{h}\binom{g}{h}\\
 &= (-1)^{l}\binom{g}{l}
\end{align*}
\end{proof}

\begin{theorem}\label{cdf}\hfill
\begin{align*}
&\sum\limits_{u=i}^{u=n}p(I=u)=\frac{N}{D}
\end{align*}
where
\begin{align*}
&N:=\sum\limits_{m=\text{\emph{max}}\{n-v_j\}}^{m=n-i}(-1)^{m} \binom{n}{m} \left( (-1)^{\text{\emph{max}}\{n-v_j\}}\binom{n-m-1}{n-\text{\emph{max}}\{n-v_j\}} + (-1)^{n-i}\binom{n-m-1}{i-1} \right) \prod\limits_{j=1}^{j=m} \binom{m}{n-v_j}\\
&D:=\prod\limits_{j=1}^{j=k}\binom{n}{v_j}
\end{align*}
\end{theorem}
\begin{proof} First note that $p(I=u)=0$ if $u>\text{min}\{v_j\}$, so the sum stops at $u=\text{min}\{v_j\}$. We apply the formula for $p(I=u)$:
\begin{align*}
&\sum\limits_{u=i}^{u=\text{min}\{v_j\}}\binom{n}{u}a(n-u, (v_1-u,\ldots, v_k-u)) = \sum\limits_{u=i}^{u=\text{min}\{v_j\}}\binom{n}{u}\sum\limits_{m=0}^{m=n-u}(-1)^{n-u+m}\binom{n-u}{m}\prod\limits_{j=1}^{j=m}\binom{m}{n-v_j} \\
&= (-1)^{n} \sum\limits_{m=0}^{m=\infty}(-1)^{m} \prod\limits_{j=1}^{j=m} \binom{m}{n-v_j} \sum\limits_{u=i}^{u=\text{min}\{v_j\}} (-1)^{u} \binom{n}{n-u} \binom{n-u}{m} \\
&= (-1)^{n} \sum\limits_{m=0}^{m=\infty}(-1)^{m} \prod\limits_{j=1}^{j=m} \binom{m}{n-v_j} \sum\limits_{u=i}^{u=\text{min}\{v_j\}} (-1)^{u} \binom{n-m}{u} \binom{n}{m} \\
&= (-1)^{n} \sum\limits_{m=0}^{m=\infty}(-1)^{m} \binom{n}{m} \prod\limits_{j=1}^{j=m} \binom{m}{n-v_j} \sum\limits_{u=i}^{u=\text{min}\{v_j\}} (-1)^{u} \binom{n-m}{u} \\
&= (-1)^{n} \sum\limits_{m=0}^{m=\infty}(-1)^{m} \binom{n}{m} \prod\limits_{j=1}^{j=m} \binom{m}{n-v_j} \left( (-1)^{\text{min}\{v_j\}}\binom{n-m-1}{\text{min}\{v_j\}} - (-1)^{i-1}\binom{n-m-1}{i-1} \right) \\
&= (-1)^{n} \sum\limits_{m=0}^{m=\infty}(-1)^{m} \binom{n}{m} \left( (-1)^{\text{min}\{v_j\}}\binom{n-m-1}{\text{min}\{v_j\}} + (-1)^{i}\binom{n-m-1}{i-1} \right) \prod\limits_{j=1}^{j=m} \binom{m}{n-v_j} \\
&= \sum\limits_{m=0}^{m=\infty}(-1)^{m} \binom{n}{m} \left( (-1)^{n-\text{min}\{v_j\}}\binom{n-m-1}{\text{min}\{v_j\}} + (-1)^{n-i}\binom{n-m-1}{i-1} \right) \prod\limits_{j=1}^{j=m} \binom{m}{n-v_j} \\
&= \sum\limits_{m=0}^{m=\infty}(-1)^{m} \binom{n}{m} \left( (-1)^{\text{max}\{n-v_j\}}\binom{n-m-1}{n-\text{max}\{n-v_j\}} + (-1)^{n-i}\binom{n-m-1}{i-1} \right) \prod\limits_{j=1}^{j=m} \binom{m}{n-v_j} \\
&= \sum\limits_{m=\text{max}\{n-v_j\}}^{m=n-i}(-1)^{m} \binom{n}{m} \left( (-1)^{\text{max}\{n-v_j\}}\binom{n-m-1}{n-\text{max}\{n-v_j\}} + (-1)^{n-i}\binom{n-m-1}{i-1} \right) \prod\limits_{j=1}^{j=m} \binom{m}{n-v_j}
\end{align*}
\end{proof}
Plots of the PMF/CDFs for some values of the parameters are shown in Figure \ref{pdfs}. The behavior of the test in an example case is illustrated in Figure \ref{illustration1case}.

\subsection{CDF generating function and incomplete beta function}
The generating function for the values of CDF$(i)$, that is with $i$ and $n$ fixed and the $v=(v_1,\ldots, v_k)$ variable, is nearly expressible as the regularized incomplete beta function $I_{x}(a,b)$ with certain arguments, establishing a strong analogy to the binomial distribution. The number of configurations with up to $i$ incidence statistic is given by the generating function:
\begin{align*}
\sum\limits_{v}&\sum\limits_{u=0}^{u=i}\binom{n}{u}a(n-u, (v_1-u,\ldots, v_k-u))t^{v} = \sum\limits_{u=0}^{u=i}\binom{n}{u} (f(t)-t_1\cdots t_k)^{n-u}(t_1\cdots t_k)^{u}\\
&= f(t)^{n} \sum\limits_{u=0}^{u=i} \binom{n}{u}\left(1 - \frac{t_1\cdots t_k}{f(t)}\right)^{n-u}\left( \frac{t_1\cdots t_k}{f(t)}\right)^{u}\\
&= f(t)^{n} \thinspace I_{1 - \frac{t_1\cdots t_k}{f(t)}}(n-i, i+1)
\end{align*}
The last equation above is a "formal" application of the expression for the CDF of a binomial distribution with $n$ trials, that is,
\begin{align*}
\sum\limits_{u=0}^{u=i} \binom{n}{u} p^{u}(1-p)^{n-u} = I_{1-p}(n-i, i+1)
\end{align*}
except that instead of the usual real parameter $p\in [0,1]$ of such a distribution, $p$ must be permitted to be equal to the power series $\frac{t_1\cdots t_k}{f(t)}$ which tabulates information across all of the different values of the parameters $v=(v_1,\ldots, v_k)$.

The total number of configurations is given by the generating function $(f(t))^{n}$, so the generating function for CDF$(i)$ is the ratio:
\begin{align*}
&f(t)^{n}I_{1 - \frac{t_1\cdots t_k}{f(t)}}(n-i, i+1)\thinspace // \thinspace f(t)^{n}
\end{align*}
Here the double division symbol $//$ means the coefficient-wise ratio of the multi-dimensional series represented by the respective generating functions. Thus, despite the analogy with the binomial distribution, the generating function for CDF$(i)$ is not literally equal to $I_{1 - \frac{t_1\cdots t_k}{f(t)}}(n-i, i+1)$.

\section{Software implementation}
\subsection{Python package}
A Python package \texttt{\href{https://pypi.org/project/coincidencetest/}{coincidencetest}} is released on PyPI. It contains a self-contained module, with no dependencies beyond the standard library, that calculates the $p$-value for the test.

\subsection{Command-line tool}
A command-line tool is distributed with \texttt{coincidencetest} that bundles together a basic, lightweight signature discovery algorithm as well as test evaluation on an input binary matrix file. This may be run in a non-interactive context on a remote server or as part of a pipeline.

\subsection{Web application}
A simple GUI performs signature discovery and evaluation in real-time after user upload of a binary matrix file. A screenshot is shown in Figure \ref{screenshot_gui}.

\subsection{Testing}
The Python package contains a test suite which verifies the $p$-value formulas (i.e. the PMF and CDF) against brute-force enumerations for several small values of the parameters, furnishing rigorous computational evidence for the main theorems in addition to the proofs.

\section{Related work}
The test turns out to specialize to the Fisher exact test\autocite{fisher1922} in the case of 2 features, $|F|=2$. The incidence statistic and the frequencies of each feature provide the same information as a $2 \times 2$ integer contingency table, and the formula for the probability value agrees with ours in this case.

The Fisher exact test has been generalized to larger, $r \times c$ contingency tables\autocite{highercontingency}. Whether such tables are regarded as pertaining to 2 categorical variables with $r$ and $c$ categories respectively, or as pertaining to pairs of binary variables, one from a list of $r$ variables and one from a list of $c$ variables, contingency table methods are second-order in that they only involve interactions between pairs of variables. Much work on exact inference generally has focused on contingency tables, with multi-dimensional generalizations appearing in the literature up to order 3 (e.g. $I\times J\times K$ tables\autocite{agresti1992survey}).

By contrast our test is inherently higher-order, depending, albeit in a simple way, on the mutual interaction of all $k$ variables. As for other higher-order methods, an investigation of the joint distribution of Bernoulli variables under certain constraints has been published\autocite{kolev2006joint}, and this may yield a test with comparable domain of applicability as our test. However, as indicated in section \ref{formuladist}, the Bernoulli context involves a different null assumption.

In Good\autocite{good1976application} a very similar generating function to our $g(s,t)$ is identified as a tabulation of the number of \emph{contingency tables} (not binary matrix configurations) with fixed column and row sums. The function is $g(-s,t)^{-1}=g(s,-t)^{-1}$ (c.f. page 1166 item 5.6 and page 1182, ``$f(\mathbf{z})$"). This connection may help to explain the appearance of the beta function in the generating function for the CDF of the incidence statistic.

\appendix
\section{Formal Concept Analysis bicluster identification\label{fca_methods}}

Formal Concept Analysis (FCA)\autocite{ganter2005introtext} studies a binary data matrix, called a \emph{formal context}, in terms of a lattice of certain patterns found in the matrix. The patterns are known as \emph{(formal) concepts}. Such a concept consists of a bicluster $(F,S)$, defined as a set of features $F$ and a set of samples $S$ for which the submatrix along $(F,S)$ consists of all 1s, which is maximal in two senses: (1) $S$ cannot be enlarged without reducing $F$, and (2) $F$ cannot be enlarged without reducing $S$.

The containment relations of the sets $F$ (respectively $S$) confer a partial ordering or lattice structure on the set of all concepts, which turns out to be complete. The maximality condition amounts to a closure condition on the sets $F$ (respectively $S$), and the whole apparatus can be formulated as a Galois correspondence between two closure systems on the full feature set and full sample set.

A straightforward recursive algorithm can be used to enumerate \emph{all} concepts in a given context\autocite{ganter2016conceptual}. This algorithm applies to any finite closure system, and it works by computing the closure of the union of any pair of previously-found closed sets.

In practice, however, data sets of intermediate size or larger furnish too many concepts for a complete enumeration to provide a useful direction of attention to important subsamples or signatures. The present work is partly motivated by this problem, as it can be used to filter signatures by significance.

\newpage

\section{Figures}
\begin{figure}[ht!]
  \centering
    \includegraphics[trim=0.0cm 20cm 0cm 0cm, width=1.00\textwidth]{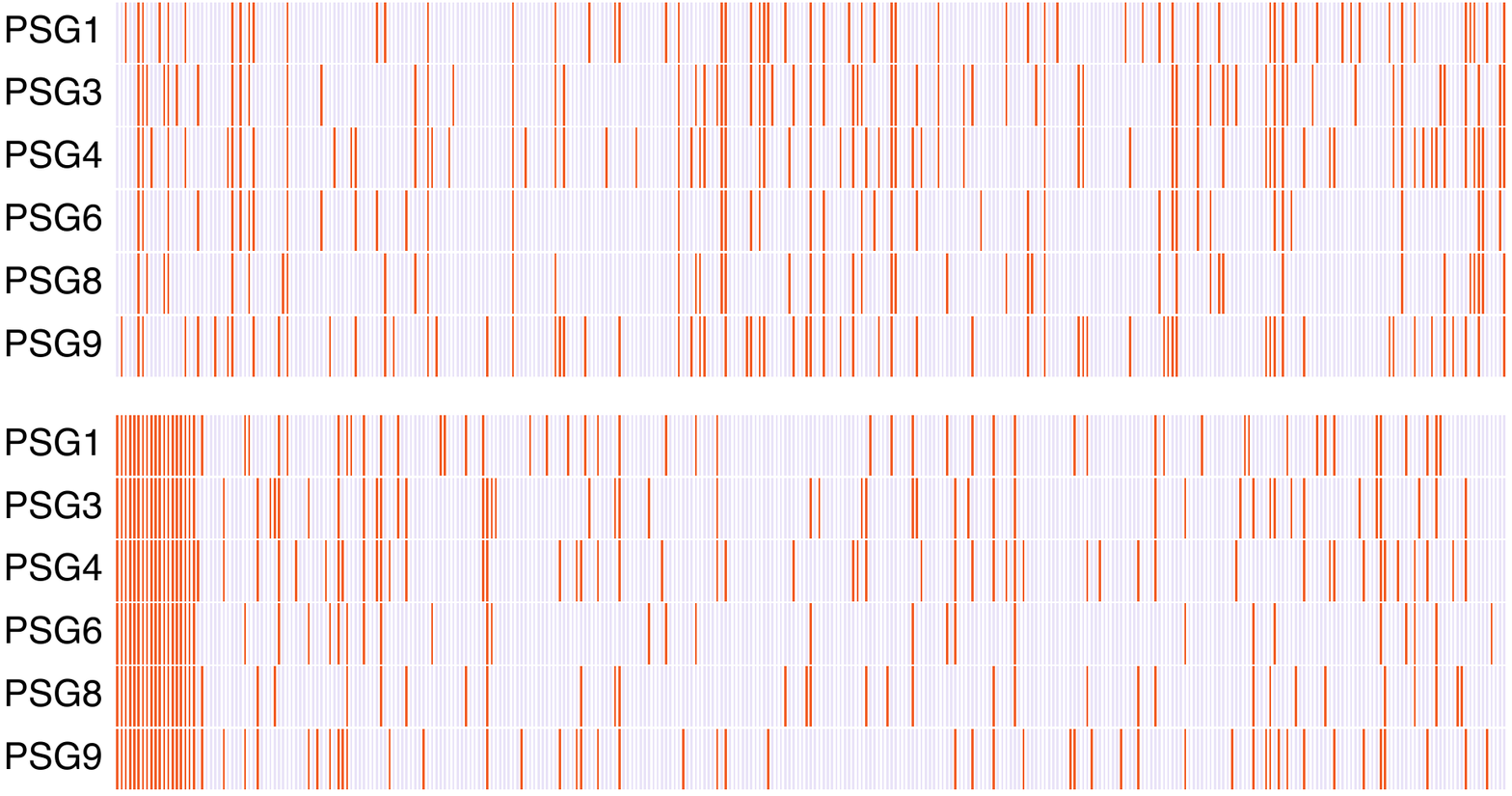}
    \caption{\label{psgfigure}(Above) The dichotomized expression of several PSG genes on 510 lung tumor samples from the TCGA-LUAD project. (Below) The same expression matrix, with the 19 samples that are positive for all features grouped together on the left. The number of positives for each feature are respectively 101, 105, 106, 73, 69, 104. The exact test for coincidence yields $p=5.1\cdot 10^{-56}$, suggesting that the PSG+ phenotype is highly statistically significant. The loci of the PSG genes are very near to each other, so this is not too surprising; it is likely that gene amplification events near this locus were the cause of the observation.}
\end{figure}

\begin{figure}[ht!]
  \centering
    \includegraphics[trim=0.0cm 14.5cm 0cm 0cm, width=1.00\textwidth]{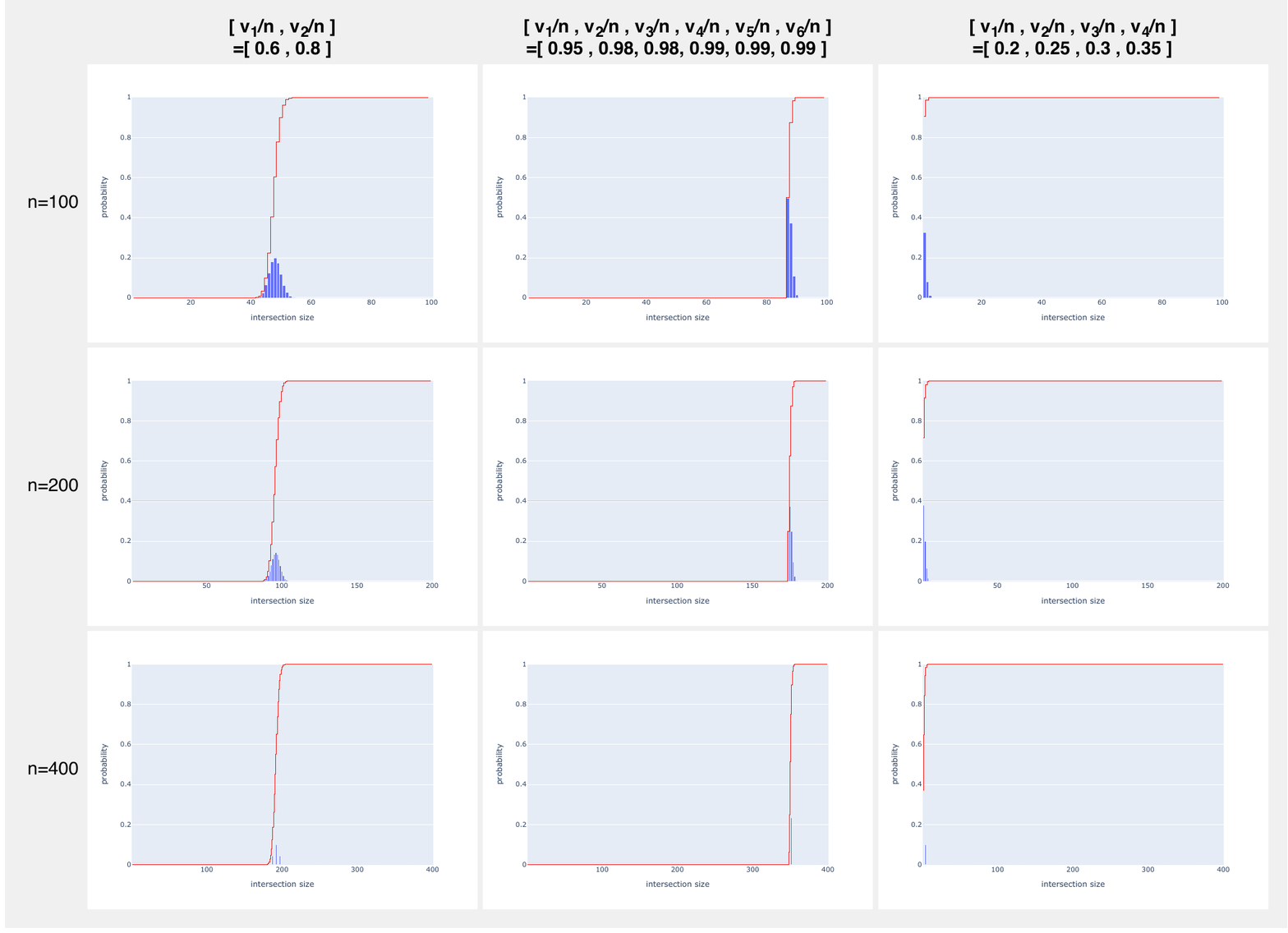}
  \caption{\label{pdfs}(Blue) The probability mass functions for the incidence statistic at several values of the set sizes $v$ and the ambient set size $n$. (Red) The cumulative distribution functions.}
\end{figure}

\begin{figure}[ht!]
  \centering
    \includegraphics[trim=0.0cm 18cm 0cm 0cm, width=1.00\textwidth]{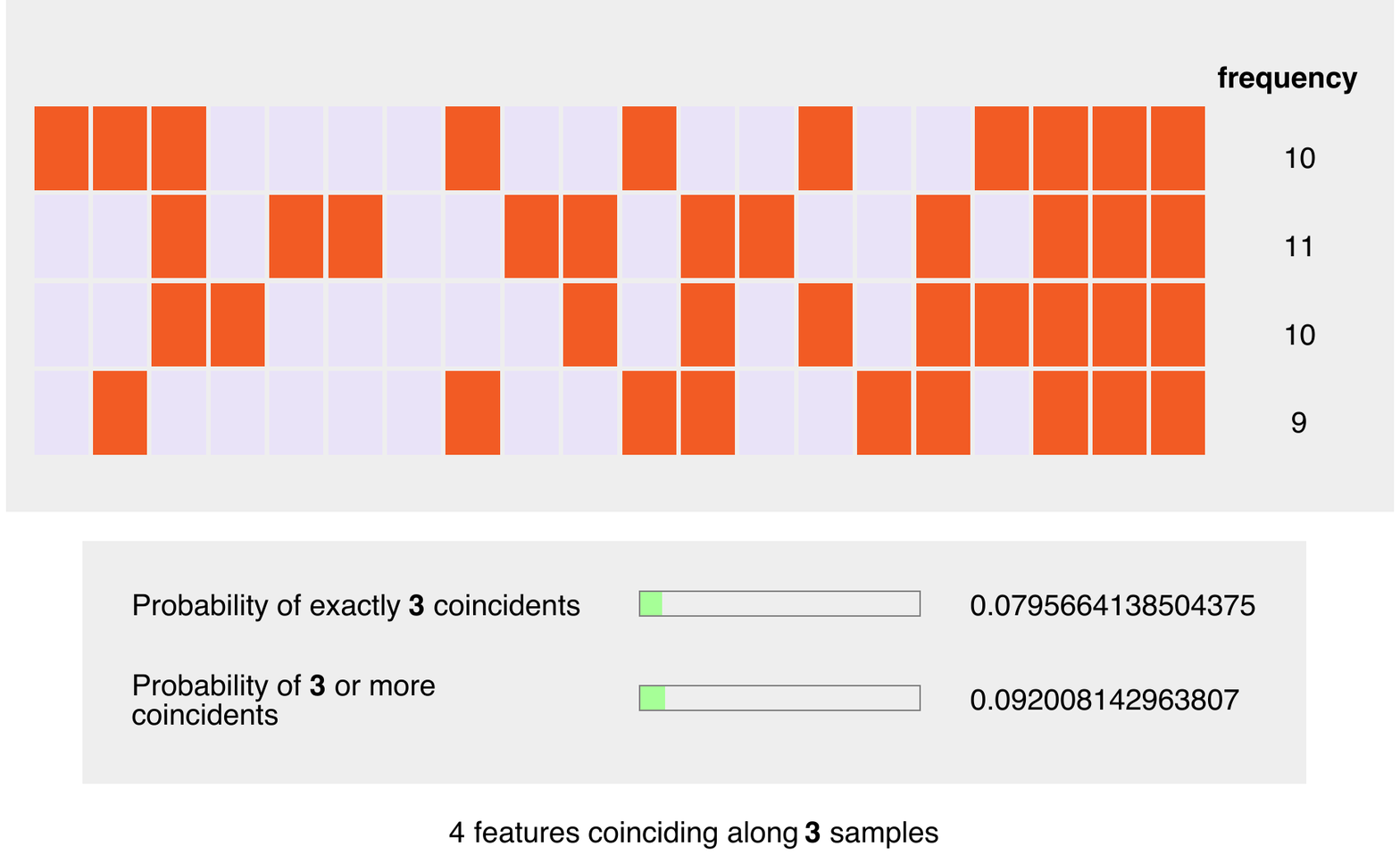}
  \caption{\label{illustration1case}Illustration of an application of the exact test for coincidence.}
\end{figure}

\begin{figure}[ht!]
  \centering
    \includegraphics[width=0.6\textwidth]{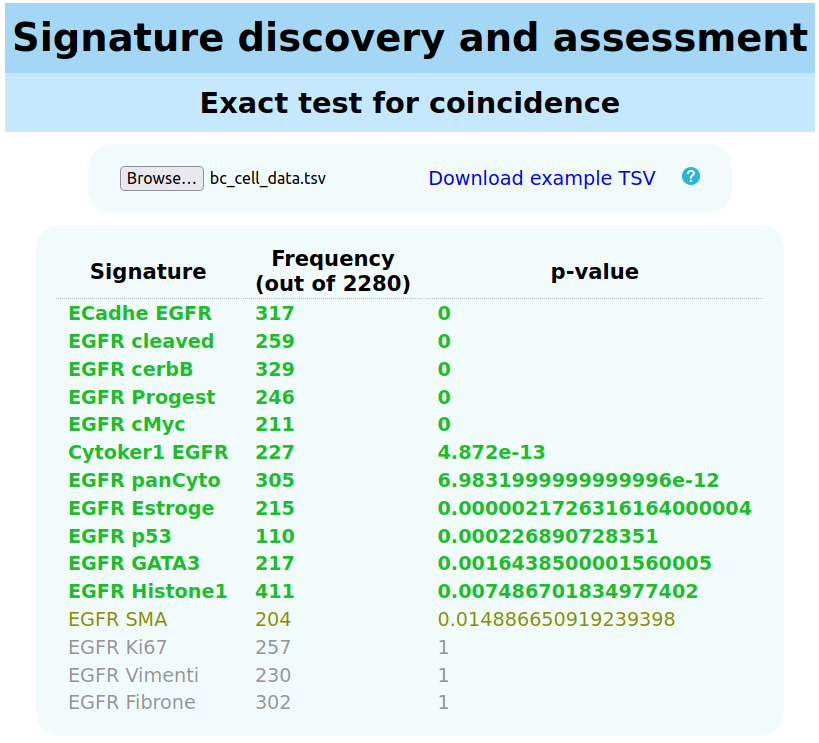}
    \caption{\label{screenshot_gui}A screenshot of the in-browser GUI.}
\end{figure}

\begin{figure}[ht!]
  \centering
    \includegraphics[width=0.95\textwidth]{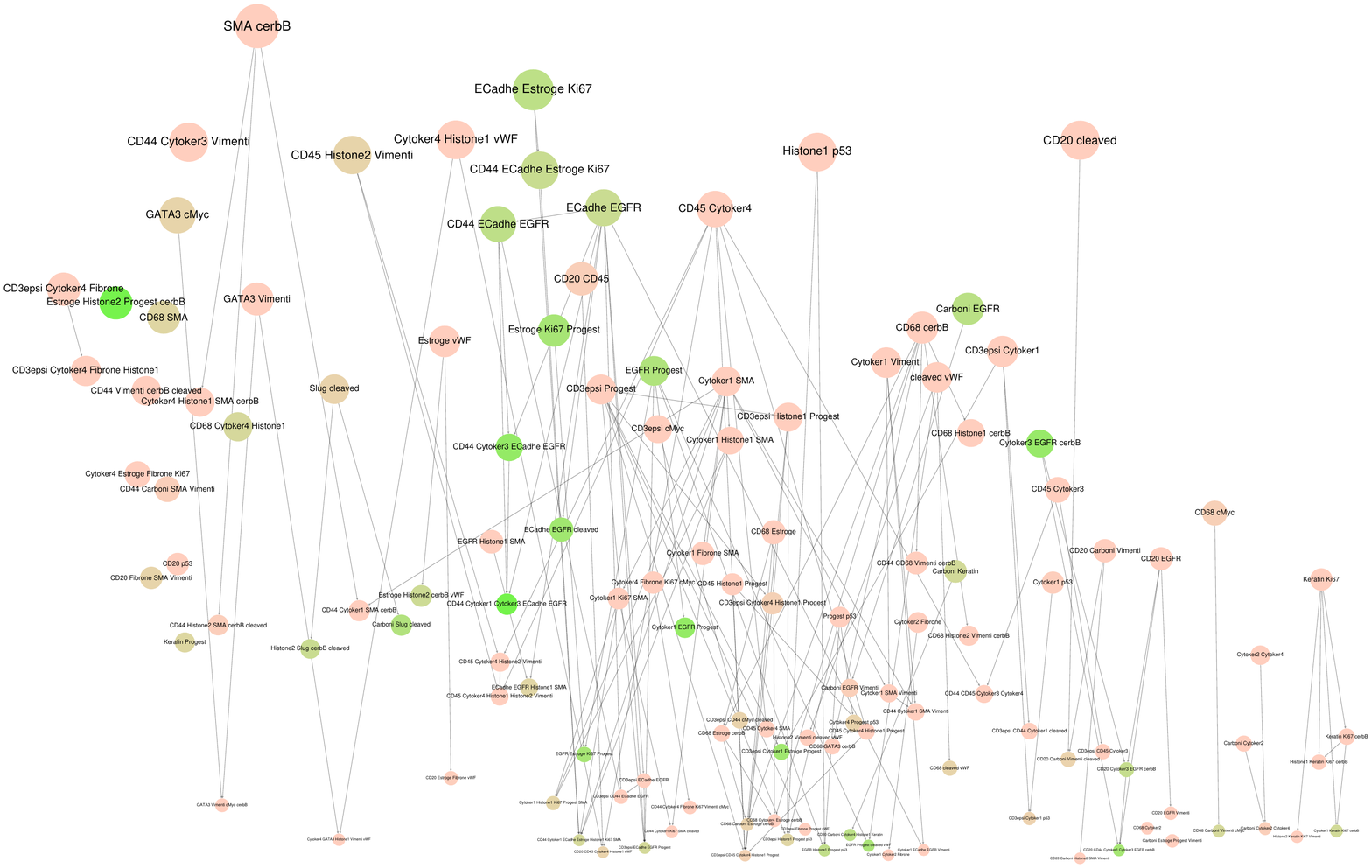}
    \caption{\label{fca_lattice}A portion of the lattice of feature subsets in cell data extracted from 29-channel multiplexed mass cytometry of breast tumor Tissue Micro Arrays (TMA)\autocite{jackson2020single}. The original data are available on Zenodo\autocite{hartland_w_jackson_2019_3518284}. For signature discovery, a random subsample was taken from the cell table. Green indicates lower $p$-value, and pink indicates higher $p$-value. The node size and vertical placement are proportional to the frequency of the sample set displaying the given signature. Only signatures with frequency between 60 and 400 (out of 2280) are shown.}
\end{figure}

\clearpage

\printbibliography

@article{fisher1922,
 ISSN = {09528385},
 URL = {http://www.jstor.org/stable/2340521},
 author = {R. A. Fisher},
 journal = {Journal of the Royal Statistical Society},
 number = {1},
 pages = {87--94},
 publisher = {[Wiley, Royal Statistical Society]},
 title = {On the Interpretation of chi-2 from Contingency Tables, and the Calculation of P},
 volume = {85},
 year = {1922}
}

@article{highercontingency,
  title={Exact tests of significance in higher dimensional tables},
  author={Zelterman, Daniel and Chan, Ivan Siu-Fung and Mielke Jr, Paul W},
  journal={The American Statistician},
  volume={49},
  number={4},
  pages={357--361},
  year={1995},
  publisher={Taylor \& Francis}
}

@article{kolev2006joint,
  title={Joint probability generating function for a vector of arbitrary indicator variables},
  author={Kolev, Nikolai and Kolkovska, Ekaterina T and L{\'o}pez-Mimbela, Jos{\'e} Alfredo},
  journal={Journal of computational and applied mathematics},
  volume={186},
  number={1},
  pages={89--98},
  year={2006},
  publisher={Elsevier}
}

@article{psgpaper,
  title={Functional network analysis reveals an immune tolerance mechanism in cancer},
  author={Mathews, James C and Nadeem, Saad and Pouryahya, Maryam and Belkhatir, Zehor and Deasy, Joseph O and Levine, Arnold J and Tannenbaum, Allen R},
  journal={Proceedings of the National Academy of Sciences},
  volume={117},
  number={28},
  pages={16339--16345},
  year={2020},
  publisher={National Acad Sciences}
}

@book{ganter2005introtext,
  title={Formal concept analysis: foundations and applications},
  author={Ganter, Bernhard and Stumme, Gerd and Wille, Rudolf},
  volume={3626},
  year={2005},
  publisher={springer}
}

@book{ganter2016conceptual,
  title={Conceptual exploration},
  author={Ganter, Bernhard and Obiedkov, Sergei and Rudolph, Sebastian and Stumme, Gerd},
  year={2016},
  publisher={Springer}
}

@article{agresti1992survey,
  title={A survey of exact inference for contingency tables},
  author={Agresti, Alan},
  journal={Statistical science},
  volume={7},
  number={1},
  pages={131--153},
  year={1992},
  publisher={Institute of Mathematical Statistics}
}

@article{good1976application,
  title={On the application of symmetric Dirichlet distributions and their mixtures to contingency tables},
  author={Good, Irving J},
  journal={The Annals of Statistics},
  volume={4},
  number={6},
  pages={1159--1189},
  year={1976},
  publisher={Institute of Mathematical Statistics}
}

@article{jackson2020single,
  title={The single-cell pathology landscape of breast cancer},
  author={Jackson, Hartland W and Fischer, Jana R and Zanotelli, Vito RT and Ali, H Raza and Mechera, Robert and Soysal, Savas D and Moch, Holger and Muenst, Simone and Varga, Zsuzsanna and Weber, Walter P and others},
  journal={Nature},
  volume={578},
  number={7796},
  pages={615--620},
  year={2020},
  publisher={Nature Publishing Group}
}

@dataset{hartland_w_jackson_2019_3518284,
  author       = {Hartland W. Jackson and
                  Jana R. Fischer and
                  Vito R.T. Zanotelli and
                  H. Raza Ali and
                  Robert Mechera and
                  Savas D. Soysal and
                  Holger Moch and
                  Simone Muenst and
                  Zsuzsanna Varga and
                  Walter P. Weber and
                  Bernd Bodenmiller},
  title        = {{The Single-Cell Pathology Landscape of Breast 
                   Cancer}},
  month        = nov,
  year         = 2019,
  publisher    = {Zenodo},
  doi          = {10.5281/zenodo.3518284},
  url          = {https://doi.org/10.5281/zenodo.3518284}
}

\end{document}